\documentclass[12pt]{amsart}
\usepackage{amsmath,amsfonts,amssymb,amsthm}
\usepackage{anysize}
\marginsize{2cm}{2cm}{2cm}{2cm}
\usepackage{graphicx}
\usepackage{color}
\usepackage{multirow}
\usepackage{ textcomp }
\usepackage{schemata}

\usepackage{enumerate}

\def\z{\mathfrak{z}}
\def\u{\mathfrak{u}}

\def\g{\mathfrak{g}}
\def\h{\mathfrak{h}}
\def\n{\mathfrak{n}}

\def\R{\mathbb{R}}

\def\ad{\operatorname{ad}}
\def\dim{\operatorname{dim}}
\def\tr{\operatorname{tr}}

\def\alt{\raise1pt\hbox{$\bigwedge$}}
\def\pint{\langle \cdotp,\cdotp \rangle }
\def\la{\langle}
\def\ra{\rangle}

\newcommand\ggo{{\mathfrak g}}

\newcommand{\iprod}{\mathbin{\lrcorner}}
\def\Ker{\operatorname{Ker}}
\def\Im{\operatorname{Im}}

\theoremstyle{plain}
\newtheorem{teo}{\bf Theorem}[section]
\newtheorem{cor}[teo]{\bf Corollary}
\newtheorem{prop}[teo]{\bf Proposition}
\newtheorem{lema}[teo]{\bf Lemma}

\theoremstyle{definition}

\theoremstyle{remark}
\newtheorem{rem}[teo]{\bf Remark}

\newcommand\aff{\mathfrak{aff}}

\title{Invariant conformal Killing forms on almost abelian Lie groups}
\author{C. Herrera}
\address{FCEFyN, Universidad Nacional de C\'{o}rdoba, Ciudad Universitaria, 5000 C\'{o}rdoba, Argentina}
\email{cecilia.herrera@unc.edu.ar}
\author{M. Origlia}
\address{FaMAF-CIEM, Universidad Nacional de C\'{o}rdoba, Ciudad Universitaria, 5000 C\'{o}rdoba, Argentina}
\email{marcos.origlia@unc.edu.ar}

\date{}

\begin{document}

\begin{abstract}
We describe completely conformal Killing or conformal Killing-Yano (CKY) $p$-forms on almost abelian metric Lie algebras. In particular we prove that if a $n$-dimensional almost abelian metric Lie algebra admits a non-parallel CKY $p$-form, then $p=1$ or $p=n-1$. In other words, any CKY $p$-form on a metric almost abelian Lie algebra is parallel for $2\leq p\leq n-2$.
Moreover, we characterize almost abelian Lie algebras admitting non-parallel CKY $p$-forms, and we classify all Lie algebras with this property up to dimension $5$, distinguishing also those cases where the associated simply connected Lie group admits lattices. 
\end{abstract}

\maketitle

\tableofcontents

\section{Introduction} The study of Killing forms began in 1951, by Yano  (see \cite{Yano1}), he defined Killing $p$-forms  on a Riemannian manifold $(M,g)$. A $p$-form  $\eta$  is Killing  if it satisfies the following equation
\begin{equation}\label{yanoequationoriginal}
\nabla\eta(X_1,X_2,\dots,X_{p+1})+\nabla \eta(X_2,X_1,\dots,X_{p+1})=0, 
\end{equation}
for all vector fields $X_1$, $X_2$,...,$X_{p+1}$ where $\nabla$ denotes the Levi-Civita connection associated to the metric $g$. In particular, any parallel $p$-form, that is $\nabla \eta=0$, is Killing.

In 1968, Tachibana renamed \eqref{yanoequationoriginal} in  \cite{Tachibana1} as the Killing-Yano equation. Then he extended the concept of Killing $2$-forms to conformal Killing $2$-forms, see \cite{Tachibana}.  During the same year, Kashiwada continued this generalization and defined conformal Killing (or conformal Killing-Yano) $p$-forms for  $p\geq 2$ (\cite{Kashiwada}). A $p$-form $\eta$  is called conformal Killing  if it exists a $(p-1)$-form $\theta$ such that the following equation is satisfied
\begin{align}\nonumber\label{Tachibanaequationoriginal}
		\nabla \eta(X_1,X_2,\dots,X_{p+1})+\nabla& \eta(X_2,X_1,\dots,X_{p+1})=2g(X_1, X_2) \left(X_3 , \dots , X_{p+1}\right)\\
		&  -\displaystyle{\sum_{i=3}^{p+1}}(-1)^i g\left(X_1,X_i\right)\theta\left(X_1, X_3 ,\dots , \widehat{X_i} ,\dots, X _{p+1} \right) \\
		& -\displaystyle{\sum_{i=3}^{p+1}}(-1)^i g\left(X_2,X_i\right)\theta\left(X_2, X_3 ,\dots, \widehat{X_i}, \dots , X _{p+1}\right) \nonumber
\end{align} 
for $X_1,X_2,\dots,X_{p+1}$ vector fields and where $\widehat{X_i}$ means that $X_i$ is omitted. We will abbreviate conformal Killing or conformal Killing-Yano forms by CKY for short, similarly KY will mean Killing or Killing-Yano. We will call \textit{strict} CKY to CKY forms which are not KY.

The interest on the study of this forms has grown up in the last 25 years, because they are considered a powerful tool in the general relativity and supersymmetric quantum field theory. We recommend to see \cite{O.Santillan} for more information.


In 2003, U. Semmelman introduced a different point of
view in \cite{Semmelmann} (see also \cite{Stepanov}), he described a conformal Killing $p$-form as a form in the kernel of  a first order elliptic differential operator. 
Equivalently, a $p$-form $\eta$ is  conformal Killing  on  a $n$-dimensional Riemannian manifold $(M,g)$ if it satisfies 
for any vector field $X$ the following equation
\begin{equation}\label{ckyManifolds}
	\nabla_X  \eta=\frac{1}{p+1}X \lrcorner\mathrm{d}\eta-\frac{1}{n-p+1}X^*\wedge \mathrm{d}^*\eta,
\end{equation}
where $X^*$ is the dual 1-form of $X$,  $\mathrm{d}^*$ is the co-differential operator 
and $\lrcorner$ is the contraction. 
If $\eta$ is co-closed ($\mathrm{d}^*\eta=0$), \eqref{ckyManifolds}  is equivalently to \eqref{yanoequationoriginal} and then $\eta$ is  a Killing  $p$-form. 
Semmelmann also named $*$-Killing $p$-forms to closed ($\mathrm{d}\eta=0$) conformal Killing forms.

There are a lots of important results in several works, we mention some of them 
\begin{itemize}
\item A conformal $p$-form ($p\neq 3, 4$) on a compact manifold with holonomy $G_2$ is parallel, \cite{Semmelmann}.
\item The dimension of the vector space of CKY
$p$-forms on a $n$-dimensional connected Riemannian manifold is at most $ \binom{n+2}{p+1} $, \cite{Semmelmann}. 
 \item There are no conformal
Killing forms on compact manifolds of negative constant sectional curvature, \cite{Semmelmann}.
\item  Descriptions and classifications of Killing forms: on compact K\"ahler manifolds \cite{Yamaguchi}, on compact symmetric spaces \cite{BMS}, on compact quaternion-K\"ahler manifolds \cite{Moroianu-semmelmann}, and on compact manifolds with holonomy $G_2$ or $Spin_7$ \cite{semmelmann2}.
\item  A conformal Killing $p$-forms on a compact Riemannian product is a sum of forms of the following types: parallel forms, pull-back of
Killing-Yano forms on the factors, and their Hodge duals, \cite{Moroianu-semmelmann-08}.
\end{itemize}

In this work we will focus on CKY $p$-forms on Riemannian Lie groups $(G,g)$, in particular when $g$ is left-invariant. 

\subsection{CKY $p$-forms on Lie groups}

Let $(G,g)$ be a Lie group $G$ endowed with a left-invariant Riemannian metric $g$. We denote by 
$\ggo$ the Lie algebra of left-invariant vector fields, which is isomorphic to $T_eG$, the tangent space of $G$ at the identity $e$ of $G$. The left-invariant condition imposed to the metric defines a natural inner product $\pint$ on $T_eG$, and reciprocally every inner product $\pint$ on $T_eG$ set a left-invariant metric on $G$. We will denote to $(\ggo,\pint)$ as the pair of those induced elements.

The left-invariant metric $g$ determines a unique Levi-Civita connection $\nabla$ on $G$, which has the following expression for $x,y,z\in \ggo$, known as the Koszul formula
\begin{equation}\label{koszul}
	2\la \nabla_xy,z\ra=\la [x,y],z\ra-\la [y,z],x\ra+\la [z,x],y\ra.
\end{equation}
In particular, $\nabla_x$ is a skew-symmetric endomorphism on $\ggo$.
We will focus on left-invariant $p$-forms $\omega$, that is for all $a\in G$, $L_a^* \omega=\omega$, where $L_a$ denotes the left-translation by $a\in G$. Such forms  define elements on $\Lambda^p\g^*$, and conversely every element on $\Lambda^p\g^*$ induces a left-invariant $p$-form on $G$. Since, $\mathrm{d} \omega$ and $\mathrm{d}^*\omega$ are left-invariant, one can study $\omega\in \Lambda^p\g^*$ satisfying \eqref{ckyManifolds}. We will say then that a $\omega$ is conformal Killing $p$-form on $\ggo$ if it is a  left-invariant conformal
Killing form on $(G, g)$, that is,
\begin{equation}\label{cky_lie_algebras}
	\nabla_x  \omega=\frac{1}{p+1}x\iprod\mathrm{d}\omega-\frac{1}{n-p+1}x^*\wedge \mathrm{d}^*\omega,
\end{equation}
for all $x\in\g$. 
The space of solutions of the CKY equation \eqref{cky_lie_algebras} is denoted by $\mathcal{CK}^p(\g,\pint)$, similarly we have  $\mathcal{K}^p(\g,\pint)$ and  $\mathcal{*K}^p(\g,\pint)$ the space of Killing $p$-forms and the space of $*$-Killing $p$-forms respectively.

The study of these left-invariant forms on $(G,g)$ with $g$ left-invariant began in \cite{BDS}, where Barberis, Dotti and Santill\'an studied Killing $2$-forms. Then in  \cite{ABD,AD,AD20}, Andrada, Barberis and Dotti continued working on conformal Killing $2$-forms and obtained some results on Lie groups with bi-invariant metric, 2-step nilpotent Lie groups and clasification of 3-dimensional metric Lie algebras admitting a conformal Killing $2$-form.
  
Later, del Barco and Moroianu  described Killing $p$-forms on 2-step nilpotent Lie groups in \cite{dBM19, dBM20}, and conformal Killing 2- or 3-forms on 2-step nilpotent
Riemannian Lie groups in \cite{dBM21}. 
Recently, in \cite{Her-Ori 22} we classified $5$-dimensional metric Lie algebras admitting CKY $2$-forms. For more details about known results and open problems on this topic we recommend \cite{Her-Ori survey}.

As mentioned most of the work done about left-invariant CKY forms on Lie groups is focused on $2$-forms or on $p$-forms on $2$-step nilpotent Lie groups. Therefore, in this paper we take the next step and we study CKY $p$-forms for arbitrary $p$ on a family of solvable Lie groups, namely almost abelian Lie groups.
In this direction, only the case $p=2$ was considered in \cite{AD}.

The paper is organized at follows: The second section is devoted to recall some importants properties about almost abelian Lie groups in general. We also show some useful properties of $p$-forms. In the third section we analyze CKY $p$-forms, and exhibit the main theorem of this work. Finally, in the fourth Section we present some explicit low dimensional examples. In particular,  we classify all Lie algebras with admitting CKY $p$-forms up to dimension $5$, distinguishing also those cases where the associated simply connected Lie group admits lattices.

\

\noindent \textbf{Acknowledgements.} Both authors were partially supported by CONICET, ANPCyT and SECyT-UNC 
(Argentina).

\

\section{Preliminaries}

\subsection{Almost abelian Lie groups}

In this article we will focus on a family of solvable Lie groups, namely, the almost abelian Lie groups.
We recall that a Lie group $G$ is said to be {\em almost abelian} if its Lie algebra $\g$ has a 
codimension one abelian ideal. Such a Lie algebra will be called almost abelian, and it can be 
written as $\g= \R f_1 \ltimes_{\ad_{f_1}} \mathfrak{u}$, where $\mathfrak u$ is an abelian ideal of 
$\g$, and $\R$ is generated by $f_1$. Accordingly, the Lie group $G$ is a semidirect product 
$G=\R\ltimes_\phi \R^d$ for some $d\in\mathbb N$, where the action is given by 
$\phi(t)=e^{t\ad_{f_1}}$. In particular, an almost abelian Lie algebra is nilpotent if and only 
if the operator $\ad_{f_1}|_{\mathfrak u}$ is nilpotent. 
The next result from \cite{Freibert} determines when two almost abelian Lie algebras are isomorphic.  
\begin{lema}\label{ad-conjugated}
	Two almost abelian Lie algebras $\g_1=\R f_1\ltimes_{\ad_{f_1}} \u_1$ and 
	$\g_2=\R f_2\ltimes_{\ad_{f_2}}\u_2$ are isomorphic if and only if there exists $c\neq 0$ such that
	$\ad_{f_1}$ and $c\ad_{f_2}$ are conjugate. 
\end{lema}

\medskip

\begin{rem}
	Note that a codimension one abelian ideal of an almost abelian Lie algebra is almost always unique (see \cite{AO18}). 
\end{rem}

\medskip

\subsubsection{Lattices in almost abelian Lie groups}
An important property of almost abelian Lie groups is that there exists a 
criterion to determine when such a Lie group admits lattices. In general, it is not easy to 
determine if a given Lie group $G$ admits a lattice. A well known restriction 
is that if this is the case then $G$ must be unimodular (\cite{Mi}), i.e. the Haar measure on $G$ 
is left and right invariant, or equivalently, when $G$ is connected, $\tr(\ad_x)=0$ for any $x$ in 
the Lie algebra $\g$ of $G$. In the case of an almost abelian Lie group we recall the following result.

\begin{prop}\label{latt}\cite{B}
	Let $G=\R\ltimes_\phi\R^d$ be an almost abelian Lie group. Then $G$ admits a lattice if and 
	only if there exists 
	$t_0\neq 0$ such that $\phi(t_0)$ can be conjugated to an integer matrix in $ \operatorname{SL} (d,\mathbb Z)$.

In this case, a lattice is given by $\Gamma=t_0\mathbb{Z}\ltimes P^{-1}\mathbb Z^d$, 
where $P\phi(t_0)P^{-1}$ is an integer matrix.
\end{prop}

\

\subsection{CKY $2$-forms}

On almost abelian Lie groups, CKY $p$-forms were analyzed only for $p=2$ in \cite{AD} and the main result states that only in dimension $3$ it is possible to find strict CKY 2-forms, otherwise the CKY $2$-form is parallel.
\begin{teo}\label{almost-abelian-2-forms}\cite[Theorem 6.2]{AD} Let $\mathfrak{g}$ be an almost abelian Lie algebra equipped with an inner product and a CKY $2$-form $\omega$.
	\begin{enumerate}
		\item If $\theta  \neq 0$ then $\dim \mathfrak{g}=3$, then $\ggo$ is isomorphic to $\mathfrak{h}_3$ or to $\aff(\R)\times \R$.
		\item If $\theta  =0$ then $\omega$ is parallel.
	\end{enumerate}	
\end{teo}

Moreover, in \cite{ABD} the authors classify metrics Lie algebras in dimension $3$ admitting CKY $2$-forms. We recall that result here.
\begin{teo}\cite[Theorem 5.1]{ABD}\label{metricas en dim 3}
	\begin{enumerate}
		\item Any inner product on $\mathfrak{h}_3$ admits a CKY $2$-form and any of these $2$-forms is strict.
		\item  Any inner product on $\aff(\R)\times \R$ admits a CKY $2$-form. Up to scaling, each of these metrics is isometric to $g_{1,t}$ for one and only one $t\geq 0$. For $t>0$, the CKY $2$-form is strict, while for $t=0$ it is parallel (see Table \ref{tabla_dim3}).
	\end{enumerate}	
	
	{\scriptsize
		\begin{table}[h!]
			\begin{tabular}{|c|c|c|c|c|}
				\hline
				\hline
				Lie algebra  & Lie bracket & metric & $\omega$ & {\scriptsize CKY, KY or P} \\ \hline \hline
				
				$\mathfrak{h}_3$& $[f_1,f_2]=f_3$&  $g_q=\left(\begin{array}{ccc} 1&0&0\\0&1&0\\0&0&q^2\end{array}\right)$, $q>0$& $f^1\wedge f^2$&CKY 
				\\ \hline
				
				\multirow{3}{*}{$\mathfrak{aff}(\mathbb{R})\times\R	$} &  \multirow{3}{*}{$[f_1,f_2]=f_2 $}&	 \multirow{3}{*}{$g_{1,t}=\left(\begin{array}{ccc} 1&0&0\\0&1+t^2&t\\0&t&1\end{array}\right)$, $t\geq 0$} &  \multirow{3}{*}{$f^1\wedge f^2$} & \multirow{3}{*}{ P if $t=0$}\\ 
				&&&&\\
				& & & &   CKY  if $t \neq 0$\\ 
				\hline			
			\end{tabular}
			\caption{CKY $2$-forms on almost abelian 3-dimensional metric Lie algebras}
			\label{tabla_dim3}
	\end{table}}
\end{teo}

\begin{rem}
	According to \cite{ABD} all these 2-forms are closed and in the case where $\omega$ is strict CKY, their Hodge dual are contact forms. 
\end{rem}

\subsection{$p$-forms on almost abelian Lie algebras}

In this subsection we introduce some useful notation and prove general results about $p$-forms (not necessarily CKY) on an almost abelian Lie algebra.
Let $\mathfrak{g}$ be an almost abelian Lie algebra equipped with an inner product $\pint$. If $\u$ denotes the codimension one abelian ideal in $\g$, let us choose $e_0 \in \g$ with $|e_0| = 1$
orthogonal to $\u$, so that we can decompose $\g$ orthogonally as
$\g=\R e_0\ltimes\u =\R e_0 \ltimes_M \R^{k}$ with $M =\ad_{e_0}|_\u$.
We decompose $M$ into $M=S+A$ where $S:\u\to\u$ and $A:\u\to\u$ denote its symmetric and skew-symmetric components with respect to $\pint$, respectively.

The Levi-Civita connection on $\g$ associated to $\pint$ is determined in \cite{Mi}:
\begin{equation}\label{LC_AlmostAbelian}
	\nabla_{e_0}e_0=0, \quad	\nabla_{e_0}u=Au, \quad	\nabla_ue_0=-Su, \quad	\nabla_uv=\la Su,v\ra e_0, 
\end{equation}
where $u,v\in\u$. 
Since $\Lambda^p\g^*=e^0\wedge\Lambda^{p-1}\u^*\oplus\Lambda^p\u^*$, then any $p$-form on $\g$ decomposes as
\begin{equation}\label{descomposicion w}
	\omega=e^0\wedge\alpha+\beta
\end{equation}
where $\alpha\in\Lambda^{p-1}\u^*, \beta\in\Lambda^p\u^*$ and $e^0$ denotes the dual $1$-form of $e_0$. In general, we will use $v^*$ to denote the dual $1$-form of $v\in\g$.

\

Let $B:\g\to\g$ be a Lie algebra morphism, then we denote by $B^*:\Lambda^p\g^*\to\Lambda^p\g^*$ the endomorphism given by
\begin{equation}\label{A*}
	B^*\omega(x_1,\dots,x_p)=\sum_{i=1}^p \omega(x_1,\dots,Bx_i,\dots,x_p),
\end{equation}
for $x_i\in\g$.
Next we have this technical result which will be useful later.

\begin{lema}\label{A antisimetrico}
	If $B:(\g,\pint)\to(\g,\pint)$ is a skew-symmetric Lie algebra morphism, then $B^*:\Lambda^p\g^*\to\Lambda^p\g^*$ given by \eqref{A*} is skew-symmetric with respect to the inner product in  $\Lambda^p\g^*$ induced by $\pint$.
\end{lema}
\begin{proof}
	Suppose $\{e_i\}_{i=1}^n$ is an orthonormal basis of $\g$ and $\{e^i\}_{i=1}^n$ the dual basis on $\g^*$. The vector space $\bigwedge^p \g$ has an inherited basis given by $\sigma=\{e_s=e_{s_1}\wedge e_{s_2}\wedge\cdots \wedge e_{s_p}: s=(s_1,\dots,s_p), s_i\in \{1,2,\cdots,n\}, s_1<s_2<\dots <s_p\}$. It is known that 
	$\left( \bigwedge ^p \g\right) ^*$ is isomorphic to the vector space of alternating multilinear functions on  $\g\times \g\times \cdots\times \g$ ($p$ copies).
	We define this isomorphism by
	$e^s\longrightarrow  \omega_s(x_1,\cdots,x_p)=\det\begin{pmatrix}
	e^{s_1}(x_1)\cdots e^{s_1}(x_p)\\
	\vdots\\
	e^{s_p}(x_1)\cdots e^{s_p}(x_p)
	\end{pmatrix},$
	and then any $p$-form $\omega$ can write as $\omega=\displaystyle \sum_{s}\omega(e_s)\omega_s$, therefore
	$\displaystyle	B^*w=\sum_{s}(B^*w)(e_s)w_s.$
	We compute now the coordinates of $[B^*]_\sigma$, that is 
	\begin{align*}
	[B^*]_{s,t}&=(B^*w_t)(e_s)=w_t(Be_{s_1},e_2,\cdots ,e_{s_p})+\cdots +w_t(e_{s_1},e_2,\cdots ,Be_{s_p})\\
	&=\det\begin{pmatrix}
	e^{t_1}(Be_{s_1})\cdots e^{t_1}(e_{s_p})\\
	\vdots\\
	e^{t_p}(Be_{s_1})\cdots e^{t_p}(e_{s_p})
	\end{pmatrix}+\cdots +\det\begin{pmatrix}
	e^{t_1}(e_{s_1})\cdots e^{t_1}(Be_{s_p})\\
	\vdots\\
	e^{t_p}(e_{s_1})\cdots e^{t_p}(Be_{s_p})
	\end{pmatrix}\\
	&=\det\begin{pmatrix}
	B_{t_1\,s_1}\cdots e^{t_1}(e_{s_p})\\
	\vdots\\
	B_{t_p\,s_1}\cdots e^{t_p}(e_{s_p})
	\end{pmatrix}+\cdots +\det\begin{pmatrix}
	e^{t_1}(e_{s_1})\cdots B_{t_1\,e_{s_p}}\\
	\vdots\\
	e^{t_p}(e_{s_1})\cdots B_{t_p\,s_p},
	\end{pmatrix}
	\end{align*}
	where $B_{i\,j}=\langle Be_i, e_j\rangle$. Now using that $B_{ij}=-B_{ji}$, we have that $[B^*]_{s,t}=-[B^*]_{t,s}$.
\end{proof}
Now we consider  $A^*:\Lambda^p\g^*\to\Lambda^p\g^*$ given by \eqref{A*} where $A$ is the skew-symmetric part of $\ad_{e_0}|_\u$ with respect to the given inner product $\pint$ on $\g$.
It is easy to see from \eqref{LC_AlmostAbelian} that $\nabla_{e_0}e^0=0$ and $\nabla_{u}e^0=-(Su)^*\in\u^*$. On the other hand, we compute $\nabla_{x}\beta$ for $\beta\in\Lambda^p\u^*$ and $x\in\g$ and we obtain the following result, which will be useful in the next sections. 

\begin{lema}\label{lema_util_beta}
	Let $\beta\in\Lambda^p\u^*$ for $p\geq 1$, then 
	\begin{enumerate}
		\item $\nabla_{e_0}\beta=-A^*\beta \in\Lambda^p\u^*$.
		\item $\nabla_u\beta=e^0\wedge Su\iprod\beta\in e^0\wedge \Lambda^{p-1}\u^*$ for any $u\in\u$.
	\end{enumerate}
\end{lema}
\begin{proof}
	Let $x_1,\dots,x_p\in\g$, then $\nabla_{e_0}\beta(x_1,\dots,x_p)=-\sum_{i=1}^{p}\beta(x_1,\dots,\nabla_{e_0}x_i,\dots,x_p)$.
	It follows from \eqref{LC_AlmostAbelian} that $\nabla_{e_0}\beta$ vanishes if $x_i=e_0$ for some $i$, therefore $\nabla_{e_0}\beta \in\Lambda^p\u^*$.
	Now we can assume $x_i\in u$ for all $i$, since $\nabla_{e_0}x_i=Ax_i$ then we have  
	$$\nabla_{e_0}\beta(x_1,\dots,x_p)=-\sum_{i=1}^{p}\beta(x_1,\dots,Ax_i,\dots,x_p)=-A^*\beta(x_1,\dots,x_p),$$ we obtain the first part.
	
	For the second part, let $u\in\u$, then $\nabla_{u}\beta(x_1,\dots,x_p)=-\sum_{i=1}^{p}\beta(x_1,\dots,\nabla_{u}x_i,\dots,x_p)$. Since $\beta\in\Lambda^p\u^*$, then we only need to consider the case when $x_1=e_0$ and $x_i\in\u$ for $i>1$. Therefore, $\nabla_{u}\beta(e_0,\dots,x_p)=\beta(Su,x_2,\dots,x_p)$ with $x_i\in\u$ for $i>1$, that is, $\nabla_u\beta=e^0\wedge Su\iprod\beta$. In particular, $\nabla_u\beta\in e^0\wedge \Lambda^{p-1}\u^*$.
\end{proof}

To close this section we have two observations concerning the differential and co-differential operators applied to a $p$-form defined on $\u$. First, we consider 
$\displaystyle{\u=\bigoplus_{\lambda\in L} \u_\lambda}$ where $\u_\lambda$ is the $\lambda$-eigenspace and $L$ denotes the spectrum of $S$,  since $S$ is a symmetric operator on $\u$. 

\begin{lema} \label{d y d*}
	Given a $p$-form  $\eta\in\Lambda^p\u^*$ then
	\begin{enumerate}
		\item $d\eta\in e^0\wedge \Lambda^{p}\u^*$, in particular $d(e^0\wedge\eta)=0$,
		\item $\eta$ is co-closed in $\g$, that is, $d^*\eta=0$. 
	\end{enumerate} 
\end{lema}
\begin{proof}
	The first part follows from the well known formula for the exterior derivative of a left-invariant form, that is,  $d\eta(x_0,x_1,\dots,x_p)=\displaystyle\sum_{j>i}(-1)^{i+j}\eta([x_i,x_j],x_0,\dots,\hat x_i,\dots, \hat x_j,\dots,x_p)$ for all $x_i\in\g$.
	For the second part, recall the formula of the co-differential operator in terms of the Levi-Civita connection
	$d^*\eta=-\sum_i e_i\iprod\nabla_{e_i}\eta,$ for $\{e_i\}$ an orthonormal basis. From Lemma \ref{lema_util_beta} we have that $e_0\iprod\nabla_{e_0}\eta=-e_0\iprod A^*\eta=0$ and $e_i\iprod\nabla_{e_i}\eta=e_i\iprod(e^0\wedge Se_i\iprod\eta)$. 
	Note that  for 	all $e_i\in\u$, $e_i\iprod(e^0\wedge Se_i\iprod\eta)$ clearly vanishes if $p=1$. Otherwise, if $p>1$ we consider the orthonormal basis $\{e_i\}$ adapted to the decomposition $\displaystyle{\u=\bigoplus_{\lambda\in L} \u_\lambda}$, therefore,  for all $e_i\in\u_\lambda$ and all $\lambda\in L$ we have that $e_i\iprod Se_i\iprod\eta=e_i\iprod \lambda e_i\iprod\eta=0.$ 
\end{proof}

\

\section{CKY $p$-forms on almost abelian Lie algebras}

In this section we study CKY $p$-forms on a almost abelian Lie algebras $\g$. We start by analyzing parallel forms on $\g$, then we focus on Killing forms, and finally we consider CKY $p$-forms. Recall $\omega=e^0\wedge\alpha+\beta$ as in \eqref{descomposicion w}.

\begin{prop}\label{w parallel}
A $p$-form $\omega=e^0\wedge\alpha+\beta$ with $\alpha\in\Lambda^{p-1}\u^*, \beta\in\Lambda^p\u^*$ is parallel if and only if $e^0\wedge\alpha$ and $\beta$ are parallel. Moreover, 
\begin{enumerate}
	\item for $p\geq 1$, $\beta\in\Lambda^p\u^*$ is parallel iff $A^*\beta=0$ and $\beta\in\Lambda^p(\Ker S)^*$
	\item for $p> 1$, $e^0\wedge\alpha$ is parallel iff $A^*\alpha=0$ and $\alpha\in\Lambda^*(\Ker S)^*\wedge\eta$ where $\eta$ is the volume form of $\Im S\subset \u$.
	\item $e^0$ is parallel iff $S=0$. 
\end{enumerate}
\end{prop}
\begin{proof}	
	First we compute $\nabla_{e_0}\omega=\nabla_{e_0}\left( e^0\wedge\alpha+\beta\right)=e^0\wedge\nabla_{e_0}\alpha+\nabla_{e_0}\beta$. It follows from Lemma \ref{lema_util_beta} that $\nabla_{e_0}\omega=0$ if and only if $\nabla_{e_0}(e^0\wedge\alpha)=e^0\wedge\nabla_{e_0}\alpha=0$ and $\nabla_{e_0}\beta=0$.
	On the other hand, we have $\nabla_{u}\omega=\nabla_{u}(e^0\wedge\alpha)+\nabla_{u}\beta=\nabla_{u}e^0\wedge\alpha+e^0\wedge\nabla_{u}\alpha+\nabla_{u}\beta$. Again, it follows from Lemma \ref{lema_util_beta} that $\nabla_{u}\omega=0$ if and only if $\nabla_{u}(e^0\wedge\alpha)=\nabla_{u}e^0\wedge\alpha=0$ and $\nabla_{u}\beta=0$.
	Combining both parts, we obtain that $\omega$ is parallel  if and only if both $e^0\wedge\alpha$ and $\beta$ are both parallel.
	
	Now to prove $(1)$, we use again Lemma \ref{lema_util_beta}, then $\beta\in\Lambda^p\u^*$ is parallel if and only if $A^*\beta=0$ and $Su\iprod\beta=0$ for all $u\in\u$. The latter is equivalent to require that  $\beta\in\Lambda^p(\Ker S)^*$.
	
	For $(2)$, $e^0\wedge\alpha$ is parallel if and only if $\nabla_{e_0}\alpha=-A^*\alpha=0$ and $\nabla_ue^0\wedge\alpha=-(Su)^*\wedge\alpha=0$. Note that $(Su)^*\wedge\alpha=0$ for all $u\in\u$ is equivalent to require that $\alpha\in\Lambda^*(\Ker S)^*\wedge\eta$ where $\eta$ is the volume form of $\Im S\subset \u$.
	
	Finally, for $(3)$ $e^0$ is parallel if and only if $\nabla_{u}e^0=-(Su)^*=0$ for all $u\in\u$, that is, $S=0$.
\end{proof}

\

Let us consider now Killing $p$-forms on $\g$. Recall that $\omega\in\Lambda^p\g^*$ is Killing if $$\nabla_x\omega=\frac{1}{p+1}x\iprod d\omega,$$ or equivalently, $x\iprod\nabla_x\omega=0$ for all $x\in\g$. In this case we have the following result about KY forms on metric almost abelian Lie algebras.

\begin{prop}\label{Killing-parallel}
	A $p$-form $\omega=e^0\wedge\alpha+\beta$ for $p\geq 1$ is Killing if and only if $e^0\wedge\alpha$ is parallel and $\beta$ is Killing. 
\end{prop}

\begin{proof}
	
We compute first $\frac{1}{p+1}e_0\iprod d\omega$, it follows from Lemma \ref{d y d*} that
$$\frac{1}{p+1}e_0\iprod d\omega=\frac{1}{p+1}e_0\iprod d(e^0\wedge\alpha+\beta)=\frac{1}{p+1}e_0\iprod d\beta\in  \Lambda^p\u^*.$$
Now, using Lemma \ref{lema_util_beta}, we compute $$\nabla_{e_0}\omega=\nabla_{e_0}(e^0\wedge\alpha)+\nabla_{e_0}\beta=\underbrace{\nabla_{e_0}e^0\wedge\alpha}_{=0}+\underbrace{e^0\wedge\nabla_{e_0}\alpha}_{\in\, e^0\wedge\Lambda^{p-1}\u^*}+\underbrace{\nabla_{e_0}\beta}_{\in\, \Lambda^{p}\u^*}.$$ Therefore,
\begin{equation}\label{kp1}
	\nabla_{e_0}\omega=\frac{1}{p+1}e_0\iprod d\omega \quad\Leftrightarrow \quad \nabla_{e_0}(e^0\wedge\alpha)=0 \quad \text{and}\quad  \nabla_{e_0}\beta=\frac{1}{p+1}e_0\iprod d\beta.
\end{equation}

Now, for $u\in\u$, we have $\nabla_{u}\omega=\nabla_{u}(e^0\wedge\alpha)+\nabla_{u}\beta=\underbrace{\nabla_{u}e^0\wedge\alpha}_{\in \,\Lambda^{p}\u^*}+\underbrace{e^0\wedge\nabla_{u}\alpha}_{=0}+\underbrace{\nabla_{u}\beta}_{\in\, e^0\wedge\Lambda^{p-1}\u^*}$
(note that the second term $\nabla_{u}\alpha$ makes sense for $p\geq2$, but that term does not even exist when $p=1$.)
On the other hand, $\frac{1}{p+1}u\iprod d\omega=\frac{1}{p+1}u\iprod d\beta\in e^0\wedge\Lambda^{p-1}\u^*$, then   
\begin{equation}\label{kp2}
	\nabla_{u}\omega=\frac{1}{p+1}u\iprod d\beta \quad\Leftrightarrow \quad\nabla_{u}(e^0\wedge\alpha)=0  \quad \text{and}\quad  \nabla_{u}\beta=\frac{1}{p+1}u\iprod d\beta.	
\end{equation}
From \eqref{kp1} and \eqref{kp2} it is easy to see that $\omega$ is Killing if and only if $e^0\wedge\alpha$ is parallel and $\beta$ is Killing.
\end{proof}

The last result suggests analyzing Killing $p$-forms $\beta$ defined on $\u$, that is, $\beta\in\Lambda^p\u^*$. Before that, we need the following technical result. Recall that we have $\displaystyle{\u=\bigoplus_{\lambda\in L} \u_\lambda}$ where $\u_\lambda$ is the $\lambda$-eigenspace and $L$ denotes the spectrum of $S$, since $S$ is a symmetric operator on $\u$. 

\begin{lema}\label{lema oscuro}
	Let $\alpha,\beta\in\Lambda^p\u^*$ for $p\geq 1$. If $Sx\iprod \alpha =  x\iprod \beta$, for all $x\in\u$, then $\displaystyle \alpha,\beta\in\bigoplus_{\lambda\in L}\Lambda^p\u_\lambda^*$.
	In particular, if $p>\dim\u_\lambda$ for all $\lambda \in L$, then $\alpha=\beta=0$.
\end{lema}

\begin{proof} 
 Let us fix $\lambda\in L$. If $L=\{\lambda\}$, there's nothing to prove. Otherwise, let us consider $\mu\in L$ with $\mu\neq \lambda$. We decompose $ \Lambda^p\u^*=  \Omega_{\lambda,\mu} \oplus \Omega_{\lambda} \oplus \Omega_{\mu} \oplus \Omega$, where
 \begin{align*}
 \Omega_{\lambda,\mu}=& \displaystyle{\bigoplus_{a+b+c=p, a,b\geq1} }\Lambda^a\u_\lambda^* \otimes \Lambda^b\u_\mu^* \otimes  \Lambda^c  \displaystyle{\oplus_{\sigma\neq\lambda,\mu} \u_\sigma^*} \\
 \Omega_{\lambda}=&  \bigoplus_{a+c=p, a\geq1} \Lambda^a\u_\lambda^*  \otimes  \Lambda^c \oplus_{\sigma\neq\mu} \u_\sigma^* \\
 \Omega_{\mu}=& \bigoplus_{b+c=p, b\geq1} \Lambda^b\u_\mu^*  \otimes  \Lambda^c \oplus_{\sigma\neq\lambda} \u_\sigma^* \\
 \Omega = & \Lambda^p \oplus_{\sigma\neq\lambda,\mu} \u_\sigma^*.
\end{align*}
Therefore, 
given $\alpha\in \Lambda^p\u^*$, we can write $\alpha=\alpha_{\lambda,\mu} + \alpha_{\lambda} + \alpha_{\mu} + \alpha_0$ with $\alpha_{\lambda,\mu}\in\Omega_{\lambda,\mu}, \alpha_{\lambda}\in\Omega_{\lambda}, \alpha_{\mu}\in\Omega_{\mu}, \alpha_0\in\Omega$. 

The condition $Sx\iprod \alpha =  x\iprod \beta$, for all $x\in\u$ reduces to $x\iprod \lambda\alpha =  x\iprod \beta$, for all $x\in\u_\lambda$ and for all $\lambda\in L$.
In particular, 
$x\iprod \lambda\alpha= x\iprod\lambda\alpha_{\lambda,\mu} + x\iprod\lambda\alpha_{\lambda} + x\iprod\lambda\alpha_{\mu} + x\iprod\lambda\alpha_0= x\iprod\lambda\alpha_{\lambda,\mu} + x\iprod\lambda\alpha_{\lambda}$, and similarly we have
$x\iprod \beta=x\iprod\beta_{\lambda,\mu} + x\iprod\beta_{\lambda} + x\iprod\beta_{\mu} + x\iprod\beta_0= x\iprod\beta_{\lambda,\mu} + x\iprod\beta_{\lambda}$. Therefore, we obtain for all $x\in\u_\lambda$, $$x\iprod\lambda\alpha_{\lambda,\mu}=x\iprod\beta_{\lambda,\mu}, \,\,  x\iprod\lambda\alpha_{\lambda}=x\iprod\beta_{\lambda}.$$ 

In the same way, we compute $y\iprod \lambda\alpha =  y\iprod \beta$ for $y\in\u_\mu$ with $\mu\neq \lambda$, and we obtain
$$y\iprod\mu\alpha_{\lambda,\mu}=y\iprod\beta_{\lambda,\mu}, \, \, y\iprod\mu\alpha_{\mu}=y\iprod \beta_{\mu},$$ 
for all $y\in\u_\mu$.
In particular, the first equation $x\iprod\lambda\alpha_{\lambda,\mu}=x\iprod\beta_{\lambda,\mu}$ for all $x\in\u_\lambda$ says that the coefficients (for some basis of $\Omega_{\lambda,\mu}$) of $\beta_{\lambda,\mu}$ are a multiple of the coefficients of $\alpha_{\lambda,\mu}$ and this multiple is $\lambda$. 
On the other hand, the second equation
$y\iprod\mu\alpha_{\lambda,\mu}=y\iprod\beta_{\lambda,\mu}$ for all  $y\in\u_\mu$
says the same with $\mu$. Since $\mu\neq \lambda$, therefore $\alpha_{\lambda,\mu}=0$ and then $\beta_{\lambda,\mu}=0$.
Since $\mu$ is arbitrary, then the statement follows.
\end{proof}

We consider now Killing $p$-forms defined on $\u$. If $p=1$, Killing $1$-forms correspond to Killing vectors, and we have:
\begin{prop}\label{1-forms Killing in u}
	Let $\beta\in\Lambda^p\u^*$, then $\beta$ is a Killing $1$-form on $\g$ if and only if its dual vector $z\in\z(\g)=\Ker(A+S)$. Moreover, $\beta$ is parallel if and only if $z\in \Ker A\cap\Ker S$.
\end{prop}
\begin{proof}
	$\beta$ is Killing if its dual vector $z$ is a Killing vector on $\g$, that is, $\ad_z$ is a skew-symmetric endomorphism of $\g$, which is equivalent to $z\in\z(\g)=\Ker(A+S)$. Moreover, $\beta$ is parallel if $z\in \Ker A\cap\Ker S$, as a consequence of Lemma \ref{lema_util_beta}. 
\end{proof}

On the other hand, Killing $p$-forms with $p\geq 2$ turn out to be parallel as the following result shows.

\begin{teo}\label{Killing-parallel_u}
	Let $\beta\in\Lambda^p\u^*$ with $p\geq2$, if $\beta$ is a Killing $p$-form on $\g$ then $\beta$ is parallel. Moreover, any Killing $p$-form $\omega=e^0\wedge\alpha+\beta$ as above is parallel.
\end{teo}
\begin{proof}
	For the case $p=2$, it is included in Theorem \ref{almost-abelian-2-forms} $(2)$. We assume now $p\geq3$, and let $\beta\in\Lambda^p\u^*$ be a Killing $p$-form on $\g$, which is equivalent to $(ae_0+x) \iprod\nabla_{ae_0+x}\beta=0$ for all $a\in\R$ and $x\in\u$. Using Lemma \ref{lema_util_beta} we have that
\begin{align*}
\left(ae_0+x \right)\iprod \nabla_{ae_0+x}\beta&=ae_0 \iprod \nabla_{x}\beta+x\iprod \nabla_{ae_0}\beta+x\iprod \nabla_{x}\beta\\ 
& = \underbrace{aS(x)\iprod \beta -  ax\iprod A^*\beta}_{\in \Lambda^{p-1}\mathfrak{u}^*} + \underbrace{x\iprod (e^0\wedge S(x)\iprod\beta)}_{\in e^0\wedge\Lambda^{p-2}\mathfrak{u}^*}\\
\end{align*}
Then, $\beta$ is Killing if and only if $Sx\iprod \beta -  x\iprod A^*\beta=0$ and $x\iprod Sx\iprod\beta=0$. Moreover, the second condition follows from the first one by contracting with $x$. Therefore, $\beta$ is Killing if and only if for all $x\in\u$ it satisfies
\begin{equation}\label{beta en u Killing}
Sx\iprod \beta =  x\iprod A^*\beta.
\end{equation}
It follows from Lemma \ref{lema oscuro} that
$$\beta=\sum_{\lambda}\beta_\lambda, \quad A^*\beta=\sum_{\lambda}\eta_\lambda, \quad\text{with} \quad  \beta_\lambda,\eta_\lambda\in\Lambda^p\u_\lambda^*.$$
Since $x\iprod\beta=x\iprod\beta_\lambda$ for $x\in\u_\lambda$, then it follows from \eqref{beta en u Killing} that $\lambda\beta_\lambda=\eta_\lambda$. 

On the other hand, $A^*\beta= \displaystyle A^*\sum_{\lambda\neq0}\beta_\lambda=\sum_{\lambda\neq0} A^*\beta_\lambda$. We claim that 
$A^*\beta_\lambda\in\Lambda^p\u_\lambda^*$ (this holds for $p\geq3$), that is, 
$\eta_\lambda=A^*\beta_\lambda$. Indeed, it follows from \eqref{A*} that $A^*\beta_\lambda\in\u^*\otimes\Lambda^{p-1}\u_\lambda^*$. Let us consider now $x\in \u$, $x_i\in\u_\lambda$ for $i=2,\dots,p$ we have that $A^*\beta_\sigma(x,x_2,\dots,x_p)=0$ for all $\sigma\neq\lambda$, since $p\geq 3$. 
Then, $A^*\beta_\lambda(x,x_2,\dots,x_p)=\sum_\sigma A^*\beta_\sigma(x,x_2,\dots,x_p)=A^*\beta(x,x_2,\dots,x_p)=0$ for any $x\in\u_\sigma$ with $\sigma\neq \lambda$,
since $A^*\beta\in\displaystyle\bigoplus_{\lambda\neq0}\Lambda^p\u_\lambda^*$, therefore, $A^*\beta_\lambda\in\Lambda^p\u_\lambda^*$ and the claim is proved.
Comparing now both expressions of $\eta_\lambda$ we obtain that $$A^*\beta_\lambda=\lambda\beta_\lambda\in\Lambda^p\u_\lambda^*,$$
which implies that $\beta_\lambda=0$ for all $\lambda\neq0$ according to 
Lemma \ref{A antisimetrico}.
Therefore, $\beta=\beta_0\in\Lambda^p\u_0^*$, and \eqref{beta en u Killing} reduces to $A^*\beta=0$, which is equivalent to $\beta$ being parallel according to Proposition \ref{w parallel} $(1)$.

Finally, if we consider a Killing $p$-form $\omega=e^0\wedge\alpha+\beta$ for $p\geq 2$, it follows from Proposition \ref{Killing-parallel} that $e^0\wedge\alpha$ is parallel and $\beta$ is Killing.  Then, we proved that $\beta$ has to be parallel, and therefore $\omega$ is parallel as well.
\end{proof}


\begin{rem}\label{*Killing-parallel}
	Recall that a $p$-form $\omega$ is called $*$-Killing if $\nabla_x\omega=\frac{1}{d-p+1}x\wedge d^*\omega$, with $d=\dim\g$ for all $x\in\g$.
	Using that the $*$-Hodge operator interchanges Killing and $*$-Killing forms, therefore the proposition above tell us that any $*$-Killing form on $\g$ is also parallel if $p\leq d-2$.
\end{rem}

Finally, we consider the most general case, that is, conformal Killing $p$-forms on $\g$. Recall this means, for all $x\in\g$, that $$\nabla_x\omega=\frac{1}{p+1}x\iprod d\omega+\frac{1}{d-p+1}x\wedge d^*\omega.$$ 
We are in conditions to state the main result of this section.

\begin{teo}\label{ck iff *ky y ky}
	Let $\omega=e^0\wedge\alpha+\beta$ be a $p$-form as above with $p\geq1$. Then $\omega$ is conformal Killing on $\g$ if and only if $e^0\wedge\alpha$ is $*$-Killing and $\beta$ is Killing.
\end{teo}
\begin{proof}
Let $\omega=e^0\wedge\alpha+\beta$ be a CKY $p$-form, as a consequence of Lemma \ref{lema_util_beta} we have that 
\begin{align*}
	\nabla_{e_0}\omega&=\underbrace{\nabla_{e_0}(e^0\wedge\alpha)}_{\in e^0\wedge\Lambda^{p-1}\mathfrak{u}^*}+ \underbrace{\nabla_{e_0}\beta}_{\in\Lambda^{p}\mathfrak{u}^*}\\
	e_0 \iprod d\omega &=e_0 \iprod \underbrace{d(e^0\wedge\alpha)}_{=0} + \underbrace{e_0 \iprod d\beta}_{\in\Lambda^{p}\mathfrak{u}^*} \\
	e_0 \wedge d^*\omega& = \underbrace{e_0 \wedge d^*(e^0\wedge\alpha)}_{\in e^0\wedge\Lambda^{p-1}\mathfrak{u}^*} +   \underbrace{e_0 \wedge d^*\beta}_{=0}
\end{align*}	
	Similarly, for any $x\in\u$ we have
\begin{align*}
	\nabla_{x}\omega&=\underbrace{\nabla_{x}(e^0\wedge\alpha)}_{\in\Lambda^{p}\mathfrak{u}^*}+ \underbrace{\nabla_{x}\beta}_{\in e^0\wedge\Lambda^{p-1}\mathfrak{u}^*}\\
	x\iprod d\omega &=x \iprod \underbrace{d(e^0\wedge\alpha)}_{=0} + \underbrace{x \iprod d\beta}_{\in e^0\wedge\Lambda^{p-1}\mathfrak{u}^*} \\
	x \wedge d^*\omega& = \underbrace{x \wedge d^*(e^0\wedge\alpha)}_{\in \Lambda^{p}\mathfrak{u}^*} +   \underbrace{x \wedge d^*\beta}_{=0}
\end{align*}
Then, $\omega$ is CKY if and only if $e^0\wedge\alpha$ is $*$-Killing and $\beta$ is Killing. 
\end{proof}

\begin{cor}
		Any conformal Killing $p$-form $\omega=e^0\wedge\alpha+\beta$, for $2\leq p\leq \dim\g-2$, is parallel. In other words, if there exist a non-parallel conformal Killing $p$-form, then $p=1$ or $p=\dim\g-1$.
\end{cor}

\begin{proof}
	From Theorem \ref{ck iff *ky y ky}, $\omega$ is conformal Killing if and only if $e^0\wedge\alpha$ is $*$-Killing and $\beta$ is Killing. Then, from Theorem \ref{Killing-parallel_u} we have that $\beta$ has to be parallel. On the other hand, if $e^0\wedge\alpha$ is $*$-Killing, then $\eta=*(e^0\wedge\alpha)\in \Lambda^{d-p}\mathfrak{u}^*$ is a $(d-p)$-Killing form on $\g$ with $d=\dim\g$.  Since $p\leq \dim\g-2$ we can use Theorem  \ref{Killing-parallel_u} again to obtain that $\eta$ (and then $*\eta$) has to be parallel, which implies that $\omega$ is parallel.
\end{proof}

\begin{rem}
	The condition $2\leq p\leq \dim\g-2$ in the Corollary above implies that $\dim\g\geq 4$. 	
	Note that for $p=2$, it was already known that there is not strict conformal Killing $2$-forms for $\dim\g\geq 4$, see Theorem \ref{almost-abelian-2-forms} $(1)$; while for $2< p\leq \dim\g-2$ it was not known.
\end{rem}

After the result above the only remaining cases to be considered are $p=1$ and $p=\dim \g-1$. 
In \cite[Proposition 2.1]{ABD} the authors prove that any CKY 1-form on $\g$ is Killing. 
On the other hand, if $\omega=e^0\wedge\alpha+\beta$ is a CKY $p$-form with $p=\dim \g-1$, then its Hodge-dual $*\omega$ is a CKY $1$-form in $\g$, which is Killing for the same reason (then $\omega$ is a $*$-Killing $p$-form). Therefore, in order to characterize CKY $p$-forms on $\g$ for $p=1$ and $p=\dim \g-1$ we need to study Killing $1$-forms.

\medskip

Let $\omega=ce^0+\beta$  be a Killing $1$-form with $c\in\R$ and $\beta\in\u^*$. According to Proposition \ref{Killing-parallel}, $ce^0$ is parallel and $\beta$ is Killing. In particular, $e_0$ is never a {\it strict} Killing vector (non-parallel Killing).
Moreover, from Proposition \ref{w parallel}, $e^0$ is parallel if and only if $S=0$.
On the other hand, from Proposition \ref{1-forms Killing in u} $\beta$ is Killing if its dual vector $z$ satisfies $z\in\z(\g)=\Ker(A+S)$; while $\beta$ is parallel if $z\in \Ker A\cap\Ker S$.
We summarize this in the next result.

\begin{prop}\label{kernels}
Let $\g$ be $d$-dimensional almost abelian Lie algebra $\g=\R e_0 \ltimes_{S+A} \R^{d-1}$ equipped with an inner product $\pint$, with $\ad_{e_0}=S+A$ where $S$ and $A$ denote its symmetric and skew-symmetric components with respect to $\pint$, respectively. Then $(\g,\pint)$ admits a strict Killing vector and a strict CKY $(d-1)$-form (which is $*$-Killing) if and only if $\Ker A\cap\Ker S\subsetneq \Ker(A+S)$. In this case any $z\in\Ker(A+S)\setminus\Ker A\cap\Ker S$ determines a strict Killing vector. 
\end{prop}

\begin{rem}
	Note that according to Proposition \ref{kernels}, if $\Ker A\cap\Ker S= \Ker(A+S)$ then any Killing vector on $(\g,\pint)$ is parallel. In particular, that occurs when $S$ or $A$ vanishes.
\end{rem}

\begin{rem}
	Given a metric Lie algebra $\g=\R e_0 \ltimes\R^{d-1}$ with an inner product $\pint$, the space of solutions of (conformal) Killing  $p$-forms satisfies:
	$\mathcal {CK}^1(\g,\pint)=\mathcal {K}^1(\g,\pint)=\z$, $\mathcal {CK}^p(\g,\pint)=\mathcal {*K}^p(\g,\pint)=\z$ for $p=d-1$, moreover
	parallel vectors and parallel $(d-1)$-forms are in correspondence with $\Ker A\cap\Ker S$.
	On the other hand, $\mathcal{CK}^p(\g,\pint)$ for $2\leq p\leq d-2$ contains only parallel forms, which are described by Proposition \ref{w parallel}.
\end{rem}

\section{Low dimensional cases}
As an application we determine in this section all possible $d$-dimensional almost abelian Lie algebras admitting an inner product such that $(\g,\pint)$ carries a strict CKY $(d-1)$-form for $d\leq 5$. We also analyze the existence of lattices in the simply connected Lie groups associated to those Lie algebras. The notation in this section comes from \cite{ABDO} and from \cite{B}.

In dimension $2$ there is only one non abelian Lie algebra, it is $\aff(\R)$, and its center is trivial, therefore for any metric on $\aff(\R)$ there are not Killing vectors according to Proposition \ref{kernels}.

In dimension $3$, strict CKY $2$-forms only occur on $\h_3$ and $\aff(\R)\times\R$ according to Theorem \ref{almost-abelian-2-forms}. This is also an easy application of Proposition \ref{kernels}. Indeed, given
$\g=\R\ltimes_M\R^2$, according to Lemma \ref{ad-conjugated}, we may assume that $M$ is in 
its canonical Jordan form, up to scaling. Then, there are three different 
possibilities for $M$:
$$M=\begin{pmatrix}       
	\lambda &    \\
	& \mu  
\end{pmatrix}, \;\;
\begin{pmatrix}       
\lambda & 1   \\
& \lambda 
\end{pmatrix},\;\;\
\begin{pmatrix}       
	a &  -b  \\
	b & a  
\end{pmatrix}.$$
Assume $\g$ is endowed with a inner product such that $(\g,\pint)$ admits a strict Killing vector, then $\g$ must have center, and the possibilities reduce to
$$M=\begin{pmatrix}       
	1 &    \\
	& 0 
\end{pmatrix}, \;\;
\begin{pmatrix}       
	0 & 1   \\
	& 0 
\end{pmatrix}.$$
They correspond to the Lie algebras $\aff(\R)\times\R$ and $\h_3$ respectively.
Moreover, any inner product on $\h_3$ admits a strict CKY $2$-form; and the inner product $g_{1,t}$ for $t>0$ (see Theorem \ref{metricas en dim 3}) on $\aff(\R)\times\R$ admits a strict CKY $2$-form.

\subsection{dimension $4$}

Given $\g=\R\ltimes_M\R^3$, according to Lemma \ref{ad-conjugated}, we may assume that $M$ is in 
its canonical Jordan form, up to scaling. In this case, there are four different 
possibilities for $M$, given in a basis $\{e_1,e_2,e_3\}$ of $\R^3$ by:
$$M=\begin{pmatrix}       
	\mu &   &  \\
	& \lambda  &   \\
	&   &  \delta 
\end{pmatrix}, \,
\begin{pmatrix}       
	\mu &   &  \\
	& \lambda  & 1  \\
	&   &  \lambda 
\end{pmatrix}, \,
\begin{pmatrix}       
	\mu &  1 &  \\
	& \mu  & 1  \\
	&   &  \mu 
\end{pmatrix}, \,
\begin{pmatrix}       
	\mu &   &  \\
	& \lambda  & -1  \\
	& 1  &  \lambda 
\end{pmatrix},$$
for $\lambda, \mu, \delta\in \R$, $\mu^2+\lambda^2+\delta^2\neq 0$.
We focus on $M$ with non-trivial center due to Proposition \ref{kernels}. On the other hand, we are only interested in unimodular ones, since that is a necessary condition for its associated simple connected Lie group to admit lattices. Finally, Lemma \ref{ad-conjugated} allow us to multiply by any scalar. Therefore, we reduce our list to:
$$M=\begin{pmatrix}       
	0 &   &  \\
	& 0  & 1  \\
	&   &  0 
\end{pmatrix},
\begin{pmatrix}       
	0 &  1 &  \\
	& 0  & 1  \\
	&   &  0
\end{pmatrix},
\begin{pmatrix}       
	1&   &  \\
	& -1  &   \\
	&   &  0
\end{pmatrix}, 
\begin{pmatrix}       
	0 &   &  \\
	& 0  & -1  \\
	& 1  &  0 
\end{pmatrix}.$$
The Lie algebras associated to those matrices are $\h_3\times\R$, $\n_4$ (nilpotent), $\mathfrak r_{3,-1}\times\R$ (completely solvable, that is, $\ad_x$ has only real eigenvalues for any $x$), and $\mathfrak r'_{3,0}\times\R$ (non-completely solvable), respectively.
We follow the notation in \cite{ABDO} and list the Lie brackets in Table \ref{tabla_dim4}. The Lie algebra $\mathfrak r'_{3,0}$ is also denoted by $\mathfrak e(2)$.
These Lie algebras in Table \ref{tabla_dim4} are all possible candidate to admit an inner product carrying a CKY. For $\h_3\times\R$ and $\n_4$ the 
inner product $\pint$ defined such that the basis $\{e_0,e_1,e_2,e_3\}$ is orthonormal satisfies that $\Ker A\cap\Ker S\subsetneq \Ker(A+S)$, where
$\ad_{e_0}=S+A$ with $S$ and $A$ the symmetric and skew-symmetric components with respect to $\pint$, respectively. For the cases $\mathfrak{r}_{3,-1}\times \R$, and $\mathfrak r'_{3,0}\times\R$ the same inner product does not satisfy $\Ker A\cap\Ker S\subsetneq \Ker(A+S)$ since $A=0$ or $S=0$ respectively, therefore we have to find a different metric on them.
		\begin{table}[h!]
			\begin{tabular}{|c|c|}
				\hline
				\hline
				Lie algebra  & Lie brackets  \\ \hline \hline
				
				$\R\times \mathfrak{h}_3$& $[e_0,e_3]=e_2$
				\\ 
				 \hline		
			$\mathfrak{n}_4$	& $[e_0,e_2]=e_1$, $[e_0,e_3]=e_2$ \\ \hline
		$\mathfrak{r}_{3,-1}\times \R$&$[e_0,e_1]= e_1$, $[e_0,e_2]=-e_2$ \\ \hline
				$\mathfrak{e}(2)\times \R$& $[e_0,e_2]=e_3$, $[e_0,e_3]=-e_2$\\
				\hline	
			\end{tabular}
			\caption{Almost abelian 4-dimensional Lie algebras}
			\label{tabla_dim4}
	\end{table}

\begin{rem}
For $\g=\mathfrak{r}_{3,-1}\times \R=\R\ltimes_M\R^3$ with 
 $M=\operatorname{diag}(1,-1,0)$ in a basis $\{e_1,e_2,e_3\}$ of $\R^3$, we consider the inner product $\pint$ such that the basis $\{e_0,e_1,e_2+e_3,e_3\}$ is an orthonormal basis, then the decomposition of $M=S+A$ into the symmetric and skew-symmetric parts with respect to this inner product is $M=\begin{pmatrix}       
 	1 &   &  \\
 	& -1  & 0 \\
 	& 1  &  0 
 \end{pmatrix}
=\begin{pmatrix}       
	1 &   &  \\
	& -1  & \frac12  \\
	& \frac12  &  0 
\end{pmatrix} + 
\begin{pmatrix}       
	0 &   &  \\
	& 0  & -\frac12  \\
	& \frac12  &  0 
\end{pmatrix}$. Then  $(\mathfrak{r}_{3,-1}\times \R,\pint)$ satisfies  Proposition \ref{kernels}, and therefore $(\mathfrak r_{3,-1}\times\R,\pint)$ admits a strict CKY $3$-form.

Similarly, for $\mathfrak r'_{3,0}\times\R$ we consider the inner product $\pint$ such that the basis $\{e_0,e_1,e_2+e_3,e_3\}$ is an orthonormal basis. In this case, $M$ decomposes with respect to this new inner product as $M=\begin{pmatrix}       
	0 &   &  \\
	& -1  & -1 \\
	& 2 &  1 
\end{pmatrix} =
\begin{pmatrix}       
	0 &   &  \\
	& -1  & \frac12 \\
	& \frac12 &  1 
\end{pmatrix}+
\begin{pmatrix}       
	0 &   &  \\
	& 0  & -\frac32  \\
	& \frac32  &  0 
\end{pmatrix}$, and this satisfies Proposition \ref{kernels}.
\end{rem}
We can summarize the $4$-dimensional case as follows:
\begin{prop}\label{4 unimodular con CK}
	Let $\g=\R e_0 \ltimes_{S+A} \R^3$ be a $4$-dimensional unimodular almost abelian Lie algebra. Then $\g$ admits an inner product $\pint$ such that $(\g,\pint)$ carries a strict CKY $3$-form if and only if $\g$ is one of the following Lie algebras: $\h_3\times\R$,
	$\n_4$, $\mathfrak r_{3,-1}\times\R$ or $\mathfrak e(2)\times\R$.
\end{prop}

For a complete clasification of left-invariant metrics on $4$-dimensional solvable Lie groups we refer to \cite{TS20}. In particular, the nilpotent ones were analyzed in \cite{Lauret}.

\begin{rem}\label{lattices dim 4}
The simple connected Lie groups associated to those Lie algebras in Proposition \ref{4 unimodular con CK} admit lattices. Indeed, for the nilpotent cases, the simply connected nilpotent Lie groups corresponding to the Lie algebras $\h_3 \times \R$ and $\n_4$ admit lattices due to Malcev’s criterion \cite{Malcev}, since these Lie algebras have rational structure constants for some basis.
The completely solvable simply connected
Lie group corresponding to $\mathfrak r_{3,-1}\times\R$ admits lattices, since the simply connected Lie group corresponding to $\mathfrak r_{3,-1}$, denoted by $Sol_3$ does (see for example \cite{Medina}, this is the group of rigid motions of Minkowski 2-space). Finally, it is easy to check using Proposition \ref{latt} that the simply connected solvable Lie group $E(2)$ (group of rigid motions of the Euclidean $2$-space) associated to the Lie algebra $\mathfrak e (2)$ admits lattices, and therefore $E(2) \times \R$ admits lattices, as well.
\end{rem}

\subsection{dimension $5$}

Let $\g=\R\ltimes_M\R^4$, using Lemma \ref{ad-conjugated} we first focus on all possible Jordan form of $M$, and we obtain the following different matrices $M$ given in the basis $\{e_1,e_2,e_3,e_4\}$ of $\R^4$:
$$M=\begin{pmatrix}
	\lambda&1&&\\&\lambda&1&\\&&\lambda&1\\&&&\lambda
\end{pmatrix},
\begin{pmatrix}
	\lambda&1&&\\&\lambda&1&\\&&\lambda&\\&&&\mu
\end{pmatrix},
\begin{pmatrix}
	\lambda&1&&\\&\lambda&&\\&&\mu&1\\&&&\mu
\end{pmatrix},
\begin{pmatrix}
	\lambda&1&&\\&\lambda&&\\&&\mu&\\&&&\delta
\end{pmatrix},
\begin{pmatrix}
	\lambda&&&\\&\mu&&\\&&\delta&\\&&&\sigma
\end{pmatrix},
$$
$$\begin{pmatrix}
	a&-b&&\\b&a&&\\&&c&-d\\&&d&c
\end{pmatrix},
\begin{pmatrix}
	a&-b&1&\\b&a&&1\\&&a&-b\\&&b&a
\end{pmatrix},
\begin{pmatrix}
	a&-b&&\\b&a&&\\&&\lambda&\\&&&\mu
\end{pmatrix}
$$

It follow from Proposition \ref{kernels} that the center of $\g$ has to be non trivial in order to admit a Killing vector, and we can also multiply by a constant according to Lemma \ref{ad-conjugated}. Finally, if we require also unimodularity, then we reduce the previous list to:

 $$M=\begin{pmatrix}
	0&1&&\\&0&1&\\&&0&1\\&&&0
\end{pmatrix},
\begin{pmatrix}
	0&1&&\\&0&1&\\&&0&\\&&&0
\end{pmatrix},
\begin{pmatrix}
	0&1&&\\&0&&\\&&0&1\\&&&0
\end{pmatrix},
\begin{pmatrix}
	0&1&&\\&0&&\\&&0&\\&&&0
\end{pmatrix}$$
if $\g$ is nilpotent, which which give rise to the Lie algebras  $\g_{5,2}$, $\n_4\times \R$, $\g_{5,1}$ or $\h_3\times\R^2$, respectively.

If $\g$ is completely solvable we obtain $\g_{5,8}^{-1}$, $\mathfrak r_{4,-\frac12}\times\R$, $\mathfrak r_{3,-1}\times\R^2$, or $\mathfrak r_{4,\mu,-1-\mu}\times\R$, determined respectively by the matrices:
$$\begin{pmatrix}
	0&1&&\\&0&&\\&&1&\\&&&-1
\end{pmatrix},
\begin{pmatrix}
	1&1&&\\&1&&\\&&-2&\\&&&0
\end{pmatrix},
\begin{pmatrix}
	1&&&\\&-1&&\\&&0&\\&&&0
\end{pmatrix},
\begin{pmatrix}
	1&&&\\&\mu&&\\&&-1-\mu&\\&&&0
\end{pmatrix} \; \mu\neq0.$$

Finally in the non-completely solvable case we have:
$$\begin{pmatrix}
	\lambda&-1&&\\1&\lambda&&\\&&-2\lambda&\\&&&0
\end{pmatrix} \; \lambda\neq0,\;
\begin{pmatrix}
	0&-1&&\\1&0&&\\&&0&\\&&&0
\end{pmatrix},$$
which determine the Lie algebras
$\mathfrak r'_{4,-2\lambda,\lambda}\times\R$ and
$\mathfrak r'_{3,0}\times\R^2$, respectively.

For the sake of completeness we include the Lie bracket of those Lie algebras obtained before.  

\begin{table}[h!]
	\begin{tabular}{|c|c|}
		\hline
		\hline
		Lie algebra  & Lie brackets  \\ \hline \hline
		$\g_{5,2}$ &  $[e_0,e_2]=e_1, [e_0,e_3]=e_2,  [e_0,e_4]=e_3$\\
		\hline
		$\n_4\times \R$ &  $[e_0,e_2]=e_1, [e_0,e_3]=e_2$ \\
		\hline
		$\g_{5,1}$ &  $[e_0,e_2]=e_1,   [e_0,e_4]=e_3$\\
		\hline
		$\h_3\times\R^2$ &  $[e_0,e_2]=e_1$\\
		\hline
		$\g_{5,8}^{-1}$ &  $[e_0,e_2]=e_1, [e_0,e_3]=e_3,  [e_0,e_4]=-e_4$\\
		\hline
		$\mathfrak r_{4,-\frac12}\times\R$ &  $[e_0,e_1]=e_1, [e_0,e_2]=e_1+e_2,  [e_0,e_3]=-2e_3$\\
		\hline
		$\mathfrak r_{3,-1}\times\R^2$ &  $[e_0,e_1]=e_1, [e_0,e_2]=-e_2$\\
		\hline
		$\mathfrak r_{4,\mu,-1-\mu}\times\R$ &  $[e_0,e_1]=e_1, [e_0,e_2]=\mu e_2,  [e_0,e_3]=-(1+\mu)e_3$\\ 
		\hline
		$\mathfrak r'_{4,-2\lambda,\lambda}\times\R$ &  $[e_0,e_1]=\lambda e_1 +e_2, [e_0,e_2]=-e_1+\lambda e_2,  [e_0,e_3]=-2\lambda e_3$\\
		\hline
		$\mathfrak r'_{3,0}\times\R^2$ &  $[e_0,e_1]=e_2, [e_0,e_2]=-e_1$\\	
		\hline	
	\end{tabular}
	\caption{Almost abelian 5-dimensional Lie algebras}
	\label{tabla_dim5}
\end{table}

Now, for any Lie algebra above, we need to find an inner product such that the decomposition of $M=A+S$ with respect to this inner product satisfies the condition in Proposition \ref{kernels}. For those Lie algebras whose $M$ is neither symmetric nor skew-symmetric, the inner product defined such that the basis $\{e_0,e_1,e_2,e_3,e_4\}$ is orthonormal satisfies Proposition \ref{kernels}.

The remaining cases, where this inner product does not work are $\mathfrak r_{3,-1}\times\R^2$, $\mathfrak r_{4,\mu,-1-\mu}\times\R$ and $\mathfrak r'_{3,0}\times\R^2$ corresponding to:
$$\begin{pmatrix}
	1&&&\\&-1&&\\&&0&\\&&&0
\end{pmatrix},
\begin{pmatrix}
	1&&&\\&\mu&&\\&&-1-\mu&\\&&&0
\end{pmatrix} \; \mu\neq0, \begin{pmatrix}
0&-1&&\\1&0&&\\&&0&\\&&&0
\end{pmatrix}.$$

For $\mathfrak r_{3,-1}\times\R^2$ the inner product such that $\{e_0,e_1,e_2+e_3,e_3,e_4\}$ is orthonormal satisfies the condition in Proposition \ref{kernels}.
For $\mathfrak r_{4,\mu,-1-\mu}\times\R$ the inner product such that $\{e_0,e_1,e_2,e_3+e_4,e_4\}$ is orthonormal satisfies that condition.
Finally, for $\mathfrak r'_{3,0}\times\R^2$  the inner product such that $\{e_0,e_1+e_2,e_2,e_3,e_4\}$ is orthonormal satisfies the condition in Proposition \ref{kernels}.
Therefore, all Lie algebras with associated Jordan form as above admit a strict CKY $4$-form (and thus strict Killing vectors as well). We can summarize this in the following result.

\begin{prop}\label{5unimodular}
	A $5$-dimensional unimodular almost abelian Lie algebra $\g=\R e_0 \ltimes_{S+A} \R^{4}$ equipped with an inner product $\pint$. Then $(\g,\pint)$ admits a CKY strict $4$-form if and only if $\g$ is one of the following Lie algebras: $\g_{5,2}$($4$-step nilpotent), $\n_4\times \R$($3$-step nilpotent), $\g_{5,1}$($2$-step nilpotent), $\h_3\times\R^2$($2$-step nilpotent); $\g_{5,8}^{-1}$, $\mathfrak r_{4,-\frac12}\times\R$, $\mathfrak r_{3,-1}\times\R^2$,
	$\mathfrak r_{4,\mu,-1-\mu}\times\R$ (completely solvable); and 
	$\mathfrak r'_{4,-2\lambda,\lambda}\times\R$, 
	$\mathfrak r'_{3,0}\times\R^2$ (non-completely solvable).
\end{prop}

For a complete clasification of left-invariant metrics on $5$-dimensional nilpotent Lie groups we refer to \cite{FN}.

\begin{rem}
	For the following cases we have that the simply connected Lie groups  associated to the Lie algebras in Proposition \ref{5unimodular} admit lattices: indeed, for  
	$\g_{5,2}$, $\n_4\times \R$, $\g_{5,1}$, $\h_3\times\R^2$ we can use Malcev' criterion; 
	for $\mathfrak r_{3,-1}\times\R^2$, $\mathfrak r'_{3,0}\times\R^2$ see Remark \ref{lattices dim 4};
	for $\mathfrak r_{4,\mu,-1-\mu}\times\R$ see \cite[Proposition 2.1]{Lee-etc} and for
	$\mathfrak r'_{4,-2\lambda,\lambda}\times\R$ see \cite{AO18}, both associated simply connected Lie groups admit lattices for countable many values of $\mu$ and $\lambda$;
	and for $\g_{5,8}^{-1}$ see \cite{B}. 
	On the other hand, the simply connected Lie group associated to $\mathfrak r_{4,-\frac12}\times\R$ does not admit lattices according to \cite[Theorem 7.1.1.]{B}, where this Lie algebra is denoted by $\g_{4,2}\times\R$.
\end{rem}

%
%


\begin{thebibliography}{99}\frenchspacing


	
	\bibitem{ABD} {\sc A. Andrada, M. L. Barberis, I. Dotti}, {\it Invariant solutions to the conformal 
		Killing–Yano equation on Lie groups}, J. Geom. Phys. \textbf{94} (2015), 199--208.
	
	\bibitem{ABDO} {\sc A. Andrada, M. L. Barberis, I. Dotti, G. Ovando}, {\it Product structures on four dimensional solvable Lie algebras}, Homology Homotopy Appl. {\bf 7} (2005), 9--37.
	
	
	\bibitem{AD} {\sc A. Andrada, I. Dotti}, {\it Conformal Killing-Yano $2$-forms}, Differential Geom. Appl. {\bf 58} (2018), 103--119. 
	
	\bibitem{AD20} {\sc A. Andrada, I. Dotti}, {\it Killing-Yano $2$-forms on $2$-step nilpotent Lie groups}, Geom. Dedicata {\bf 212} (2021), 415--424. 
	

	
	\bibitem{AO18} {\sc A. Andrada, M. Origlia}, {\it Lattices on almost abelian Lie groups}, Manuscripta Math. \textbf{155} (2018), 389--417.
	
	\bibitem{BDS} {\sc M. L. Barberis, I. Dotti, O. Santill\'an}, {\it The Killing-Yano equation on Lie groups},
	Class. Quantum Grav.  \textbf{29} (2012), 065004. 10pp.
	
\bibitem{BMS}	{\sc F. Belgun, A. Moroianu, U. Semmelmann}, {\it Killing forms on symmetric spaces}, Differential Geom.	Appl. \textbf{24} (2006), 215--222.
	
	\bibitem{B} 
	{\sc C. Bock}, On low-dimensional solvmanifolds, \textit{Asian J. Math.} \textbf{20} (2016), 
	199--262.
	
	
	
	\bibitem{dBM19}
	{\sc V. del Barco, A. Moroianu}, {\it Killing forms on $2$-step nilmanifolds}, J. Geom. Anal.  \textbf{31} (2021), 863--887.
	
	\bibitem{dBM20}
	{\sc V. del Barco, A. Moroianu}, {\it Higher degree Killing forms on $2$-step nilmanifolds}, J. Algebra  \textbf{563} (2020), 251--273.
	
	\bibitem{dBM21}
	{\sc V. del Barco, A. Moroianu}, {\it Conformal Killing forms on $2$-step nilpotent Riemannian Lie groups}, Forum Math. \textbf{33} (2021), 1331--1347.  
	
	
	\bibitem{FN}
	{\sc A. Figula, P. Nagy}, {\it Isometry classes of simply connected nilmanifolds}, J. Geom. Phys. \textbf{132} (2018), 370--381.
	
	\bibitem{Freibert}
	{\sc M. Freibert},	{\it Cocalibrated structures on Lie algebras with a codimension one Abelian ideal},	Ann. 	Glob. Anal. Geom. \textbf{42} (2011)
	DOI: 10.1007/s10455-012-9326-0
	
	
	\bibitem{Her-Ori 22} {\sc A. Herrera, M. Origlia}, {\it Invariant Conformal Killing–Yano 2-Forms
	on Five-Dimensional Lie Groups},  J. Geom. Anal. \textbf{32} (2022), 209--74.

	\bibitem{Her-Ori survey} {\sc A. Herrera, M. Origlia}, {\it A Survey on invariant conformal Killing forms on Lie groups}, arXiv:2312.16601.
	
	
	
	\bibitem{Kashiwada}
	{\sc T. Kashiwada}, {\it On conformal Killing tensor}, Natur. Sci. Rep. Ochanomizu
	Univ. \textbf{19} (1968), 67--74.
	
	\bibitem{Lauret}
	{\sc J. Lauret}, {\it 	Homogeneous Nilmanifolds of Dimensions 3 and 4},
	Geometriae Dedicata \textbf{68} (1997), 145--155.
	
	\bibitem{Lee-etc}
	{\sc J.B. Lee, K.B. Lee, J. Shin, S. Yi}, {\it Unimodular groups of type $R^3\rtimes R$.} J. Korean	Math. Soc. \textbf{44} (2007), 1121--1137.
	
	
	\bibitem{Malcev}
	{\sc A. Malcev}, {\em On solvable Lie algebras}, Bull. Acad. Sci. URSS. Sr. Math. [Izvestia Akad.	Nauk SSSR] \textbf{9} (1945), 329--356. 
	
	\bibitem{Medina} 
	{\sc A. Medina, P. Revoy}, {\it Lattices in symplectic Lie groups.} J. Lie Theory {\bf 17}, (2007), 27--39.
	
	\bibitem{Mi} {\sc J. Milnor}, {\em Curvature of left-invariant metrics
		on Lie groups}, Adv. Math. {\bf 21} (1976), 293--329.
	
	\bibitem{Moroianu-semmelmann} 
	{\sc A. Moroianu, U. Semmelmann}, {\it Killing forms on quaternion-K\"ahler manifolds}, Ann. Global Anal.
	Geom. \textbf{28} (2005), 319--335; erratum \textbf{34} (2008), 431--432.

\bibitem{semmelmann2} {\sc U. Semmelmann}, {\it Killing forms on $G_2$ and $Spin_7$-manifolds}, J. Geom. Phys. \textbf{56} (2006), no. 9, 1752--1766.	
	
	
	
	\bibitem{Moroianu-semmelmann-08} 
	{\sc A. Moroianu, U. Semmelmann}, {\it Twistor forms on Riemannian products}, J. Geom. Phys. \textbf{58} (2008), 1343--1345.
	
	
%
%

\bibitem{O.Santillan} {\sc O. Santill\'an}, {\it Hidden symmetries and supergravity solutions}, J. Math. Phys. \textbf{53},  043509 (2012).

%
%
%
	
	\bibitem{TS20} {\sc T. Sukilovi\'c}, {\it Classification of left-invariant metrics on  $4$-dimensional solvable Lie groups}, Theoretical and Applied Mechanics. \textbf{47} (2020), 181--204. 
	
	
	
		
	\bibitem{Semmelmann}
	{\sc U. Semmelmann}, {\it Conformal Killing forms on Riemannian manifolds}, Math. Z. \textbf{245} (2003), 503--527.
	
	\bibitem{Stepanov} {\sc S.E. Stepanov}, {\it The vector space of conformal Killing forms on a Riemannian manifold}, J. Math. Sci. \textbf{110} (2002), 2892--2906.
	
		\bibitem{Tachibana1} 
	{\sc S. Tachibana}, { \it On  Killing tensors in a Riemannian space}, Tohoku Math. J. \textbf{20} (1969), 257--264.
	


	
		
	
	
	
	
	
	\bibitem{Tachibana} 
	{\sc S. Tachibana}, { \it On conformal Killing tensors on Riemannian manifolds}, Tohoku Math. J. \textbf{21} (1969), 56--64.
	
	\bibitem{Yamaguchi} {\sc S. Yamaguchi}, {\it On a Killing $p$-form in a compact Kählerian manifold}, Tensor (N. S.) \textbf{29} (1975), no. 3, 274--276.
	
	\bibitem{Yano1}
	{\sc K. Yano}, {\it Some remarks on tensor fields and curvature}, Ann. of Math. \textbf{55} (1952), 328--347.
	
	
\end{thebibliography}
\end{document}